\documentclass[a4paper]{amsart}
\usepackage[latin1]{inputenc}
\usepackage{amssymb}
\usepackage{amsmath}
\usepackage{mathrsfs}
\usepackage{eufrak}
\usepackage{amsthm}
\usepackage{amsfonts}
\usepackage{textcomp}
\usepackage{graphicx}
\usepackage[pdftex]{color}
\usepackage{paralist}
\usepackage[shortlabels]{enumitem}
\usepackage{hyperref}
\usepackage{comment}
\usepackage[arrow, matrix, curve]{xy}

\newtheorem*{corollary*}{Corollary}
\newtheorem*{conjecture*}{Conjecture}
\newtheorem*{question*}{Question}
\newtheorem{theorem}{Theorem}[section]
\newtheorem*{theorem*}{Theorem}

\newtheorem{corollary}[theorem]{Corollary}
\newtheorem{lemma}[theorem]{Lemma}
\newtheorem{proposition}[theorem]{Proposition}

\newtheorem*{claim*}{Claim}

\theoremstyle{definition}

\newtheorem{example}[theorem]{Example}

\theoremstyle{remark}

\numberwithin{equation}{theorem}

\makeatletter
\renewcommand*\env@matrix[1][\
arraystretch]{%
  \edef\arraystretch{#1}%
  \hskip -\arraycolsep
  \let\@ifnextchar\new@ifnextchar
  \array{*\c@MaxMatrixCols c}}
\makeatother

\begin{document}

\title{On reflexive simple modules in Artin algebras}
\date{\today}



\author{Ren\'{e} Marczinzik}
\address{Institute of algebra and number theory, University of Stuttgart, Pfaffenwaldring 57, 70569 Stuttgart, Germany}
\email{marczire@mathematik.uni-stuttgart.de}

\begin{abstract}
Let $A$ be an Artin algebra. It is well known that $A$ is selfinjective if and only if every finitely generated $A$-module is reflexive. In this article we pose and motivate the question whether an algebra $A$ is selfinjective if and only if every simple module is reflexive. We give a positive answer to this question for large classes of algebras which include for example all Gorenstein algebras and all QF-3 algebras.
\end{abstract}

\maketitle
\section{Introduction}
Let an algebra $A$ always be a non-semisimple Artin algebra and modules will be finitely generated right modules if nothing is stated otherwise. In this article $(-)^{*}$ will denote the functor $Hom_A(-,A) : mod-A \rightarrow mod-A^{op}$. 
Recall that a module $M$ is called \emph{reflexive} in case the canonical evaluation map $f_M : M \rightarrow M^{**}$ is an isomorphism.
A well known characterisation of $A$ being selfinjective is that each $A$-module is reflexive, see for example \cite{L} theorem 15.11. and corollary 19.44. It is thus a natural question whether being selfinjective can be detected by the property of a smaller and possibly finite class of modules being reflexive. A natural candidate for a small and finite class of modules is the class of simple modules. This leads us to the following question:
\begin{question*}
Let $A$ be an Artin algebra such that every simple $A$-module is reflexive. Is $A$ selfinjective?

\end{question*}
In this article we give a positive answer to this question for a large class of algebras.
Recall that the Gorenstein symmetry conjecture states that the injective dimension of the left regular module coincides with the injective dimension of the right regular module. This conjecture is open and it holds for example for all algebras of finite finitistic dimension.
Our main result can be stated as follows:
\begin{theorem*}
Let $A$ be an Artin algebra that satisfies the Gorenstein symmetric conjecture and such that the injective envelope $I(A)$ of the regular module $A$ has finite projective dimension and assume that each simple $A$-module is reflexive. Then $A$ is selfinjective.
\end{theorem*}
This provides a positive answer to our question for a large class of algebras including all Gorenstein algebras and thus also for all algebras of finite global dimension.
We give also a positive answer to the question for other classes of algebras such as QF-2 algebras, QF-3 algebras, local algebras and algebras $A$ with $A \cong A^{op}$. The author is thankful to Steffen K\"onig for helpful comments.
\section{Preliminaries}
We assume that all algebras in this article are basic Artin algebras and modules are finitely generated right modules when not stated otherwise. Note that our assumption that the algebras are basic is no restriction in generality since the notions we work with are invariant under Morita equivalences.
We assume that the reader is familiar with the basics of representation theory of Artin algebras as explained for example in \cite{ARS}. For a module $M$, $Tr(M)$ will denote the transpose of $M$, $D(M)$ the natural duality of an Artin algebra applied to $M$ and $\tau$ and $\tau^{-1}$ denote the Auslander-Reiten translate and its inverse. $J$ will denote the Jacobson radical of an algebra $A$.

We call a module $M$ an \emph{$i$-syzygy module} in case $M \cong \Omega^i(N)$ for some module $N$ and $i \geq 0$. We define $\Omega^i(mod-A)$ for $i \geq 0$ to be the full subcategory of modules that are $i$-syzygy modules or projective modules.
An algebra $A$ is called a \emph{Gorenstein algebra} in case the injective dimension of the left regular module coincides with the injective dimension of the right regular module and both dimensions are finite. The injective dimension of the right regular module is called the \emph{Gorenstein dimension}. For example every algebra of finite global dimension is a Gorenstein algebra and the Gorenstein dimension coincides with the global dimension in this case. The Gorenstein projective dimension of a module $M$ over a Gorenstein algebra $A$ is defined by $Gpd(M):= sup \{ i \geq 0 | Ext_A^i(M,A) \neq 0 \}$. The \emph{global Gorenstein dimension} of an algebra $A$ is defined as the supremum of Gorenstein projective dimension of all modules and it is shown to be equal to the Gorenstein dimension in \cite{Che} corollary 3.2.6. For the definition of the Gorenstein projective dimension for general algebras and more information on Gorenstein homological algebra we refer to \cite{Che}. For algebras of finite global dimension, the Gorenstein projective dimension of a module coincides with the projective dimension of the module.
The \emph{finitistic dimension} of an algebra is defined as the supremum of projective dimensions of modules with finite projective dimension. It is an open problem whether the finitistic dimension is always finite for Artin algebras. The Gorenstein symmetry conjecture states that for an Artin algebra $A$ one has that the injective dimension of the left regular module coincides with the injective dimension of the right regular module. The Gorenstein symmetry conjecture is open as well and it is true for algebras having finite finitistic dimension, which includes for example all Gorenstein algebras or all algebras of finite representation type. We refer to the conjectures section in the book \cite{ARS} for more on this with references.
QF-3 algebras are defined as algebras such that the injective envelope of the regular module is projective. An algebra is QF-3 if and only if its opposite algebra is QF-3. A QF-2 algebra is defined by the property that the socle of each indecomposable projective module is simple. The class of QF-3 algebras is very large and includes for example all Nakayama algebras or algebras that have the double centraliser property with respect to a minimal faithful projective-injective module such as Schur algebras $S(n,r)$ for $n \geq r$. For more on QF-2 and QF-3 algebras we refer for example to \cite{AnFul}.
\section{On simple reflexive modules in Artin algebras}
Recall that a module $M$ over an algebra $A$ is called \emph{reflexive} in case the canonical evaluation map $f_M :M \rightarrow M^{**}$ is an isomorphism, where $f_M(m)(g)=g(m)$ for $m \in M$, $g \in M^{*}$ and $(-)^{*}$ denotes the functor $Hom_A(-,A)$. In case the canonical evaluation map $f_M$ is injective for a module $M$, $M$ is called \emph{torsionless}. It is well known that a module $M$ is torsionless if and only if $M$ is a submodule of a projective module.
We will need the following characterisation of torsionless and reflexive modules:
\begin{proposition}
\begin{enumerate}
\item A module $M$ is torsionless if and only if $Ext_A^1(D(A),\tau(M))=0$.
\item A module $M$ is reflexive if and only if $Ext_A^i(D(A),\tau(M))=0$ for $i=1,2$.
\end{enumerate}
\end{proposition}
\begin{proof}
See for example \cite{ARS}, corollary 3.3. in chapter IV.
\end{proof}
The previous proposition also shows that a module is torsionless (reflexive) if and only if every of its direct summands is torsionless (reflexive).
\begin{proposition} \label{2syzy}
Let $M$ be an indecomposable non-projective module. Then $M^{*}$ is isomorphic to $\Omega^{2}(Tr(M))$.
\end{proposition}
\begin{proof}
This is a special case of \cite{GH} proposition 4.8.2. 
\end{proof}
The second part of the following lemma shows that a module is reflexive if and only if $M \cong M^{**}$. While the proof is very elementary, we did not find this result in the literature and we do not know whether this result holds for arbitrary rings and finitely generated modules.
\begin{lemma}
Let $A$ be an Artin algebra and $M$ a finitely generated $A$-module.
\begin{enumerate}
\item An indecomposable non-projective module with $M \cong M^{**}$ is a 2-syzygy module.
\item $M$ is reflexive if and only if $M \cong M^{**}$.
\end{enumerate}
\end{lemma}
\begin{proof}
\begin{enumerate}
\item  We have $M^{*}=\Omega^{2}(Tr(M))$ by \ref{2syzy}, which can not have non-zero projective summands (or else $M \cong M^{**}$ would have non-zero projective summands). Thus $M^{**} = \Omega^2(Tr(\Omega^{2}(Tr(M))))$ is a 2-syzygy module.
\item In case $M$ is reflexive we have $M \cong M^{**}$ by definition.
Now assume that we have $M \cong M^{**}$. In order to prove the lemma, we can assume that $M$ is non-projective and indecomposable (since a module is reflexive if and only if every of its summands are reflexive). By the first part, $M$ is a 2-syzygy module and thus a submodule of a projective module. Thus $M$ is torsionless and $f_M$ is injective. But since $M$ and $M^{**}$ have the same length, $f_M$ must be an isomorphism.
\end{enumerate}
\end{proof}

\begin{lemma} \label{chenlemma}
Let $0 \rightarrow L \rightarrow M \rightarrow N \rightarrow 0$ be a short exact sequence.
Then the following holds:
\begin{enumerate}
\item $Gpd(N) \leq max(Gpd(M),1+Gpd(L))$.
\item $Gpd(L) \leq max(Gpd(M),Gpd(N)-1)$.
\item $Gpd(M) \leq max(Gpd(L),Gpd(N))$.

\end{enumerate}
\end{lemma}
\begin{proof}
See \cite{Che}, corollary 3.2.4.
\end{proof}

\begin{lemma} \label{gorproprop}
Let $A$ be a Gorenstein algebra of Gorenstein dimension $g$. Then the Gorenstein dimension of $A$ is equal to the supremum of the Gorenstein projective dimensions of the simple $A$-modules.
\end{lemma}
\begin{proof}
By \cite{Che}, corollary 3.2.6. the Gorenstein dimension of $A$ is equal to the supremum of Gorenstein projective dimensions of the $A$-modules. Especially: The Gorenstein projective dimension of every simple module is bounded by $g$. Now every module $M$ has a composition series and it is thus build from extensions by simple modules. We show that there exists at least one simple module of Gorenstein projective dimension $g$. Assume to the contrary that each simple module has Gorenstein projective dimension at most $g-1$. We will show that this would imply that every module has Gorenstein projective dimension at most $g-1$ and thus give a contradiction. We use induction on the length of a module.
Length one modules are exactly the simple modules and their Gorenstein projective dimension is bounded by $g-1$. Now assume that all modules of length at most $l-1$ for some $l \geq 2$ have Gorenstein projective dimension at most $g-1$, then we show that also all modules of length $l$ would have Gorenstein projective dimension at most $g-1$.
Let $M$ be a module of length $l$ and $N$ be a maximal submodule of $M$.
Then there is an exact sequence from the inclusion of $N$ into $M$ with simple cokernel $S$:
$$0 \rightarrow N \rightarrow M \rightarrow S \rightarrow 0.$$
By \ref{chenlemma} and induction hypothesis, we have $Gpd(M) \leq max(Gpd(N),Gpd(S)) \leq g-1$. This finishes the induction and shows that there must be at least one simple modules with Gorenstein projective dimension equal to $g$.

\end{proof}
\begin{theorem}
Let $A$ be an Artin algebra that satisfies the Gorenstein symmetry conjecture. Assume furthermore that the injective envelope $I(A)$ of the regular module $A$ has finite projective dimension and assume that each simple $A$-module is reflexive. Then $A$ is selfinjective.

\end{theorem}
\begin{proof}
We first prove the theorem in case $A$ is a Gorenstein algebra. Let $g>0$ (which means that $A$ can not be selfinjective) denote the Gorenstein dimension of $A$.
Assume every simple module $S$ is reflexive. Then each such $S$ is 2-syzygy module by \ref{2syzy} and thus there exists an exact sequence:
$$0 \rightarrow S \rightarrow P \rightarrow M \rightarrow 0,$$
where $P$ is projective.
By \ref{chenlemma}, this shows that $Gpd(S) \leq max(Gpd(P),Gpd(M)-1) \leq g-1$, since the global Gorenstein projective dimension is $g$ and $Gpd(P)=0$. This proves that every simple $A$-module has Gorenstein projective dimension at most $g-1$, and so by \ref{gorproprop} $A$ has Gorenstein dimension at most $g-1$, which is a contradiction. Thus $A$ has to be selfinjective.
Now assume that $A$ is such that the injective envelope $I(A)$ of the regular module $A$ has finite projective dimension and every simple $A$-module is reflexive.
Since $S$ is reflexive, we necessarily have $S^{*} \cong Hom_A(S,A) \neq 0$ or else we would also have $S \cong S^{**}=(S^{*})^{*}=0^{*}=0$.
Now $Hom_A(S,A) \neq 0$ for every simple module $S$, shows that the socle of $A$ contains every simple module at least once as a direct summand. This shows that every indecomposable injective module is a direct summand of $I(A)$. By assumption, $I(A)$ has finite projective dimension and thus each indecomposable injective module has finite projective dimension. Since we assume that $A$ satisfies the Gorenstein symmetry conjecture, this shows that $A$ is Gorenstein and thus $A$ is selfinjective by the first part of the proof.
\end{proof}

We can give a proof of a weaker version of the previous theorem when we do not include the assumption that the Gorenstein symmetry conjecture holds for the algebra. The proof relies on the following result of Auslander and Reiten:
\begin{theorem} \label{AusReiresult}
Let $A$ be an Artin algebra. Then the subcategory $\Omega^1(mod-A)$ is closed under extensions if and only if the injective envelope $I(A)$ of the regular module $A$ has projective dimension at most one.
\end{theorem}
\begin{proof}
See \cite{AR}, theorem 0.1.
\end{proof}
\begin{proposition}
Let $A$ be an Artin algebra such that the injective envelope $I(A)$ of the regular module $A$ has projective dimension at most one and such that each simple module is reflexive. Then $A$ is selfinjective.
\end{proposition}
\begin{proof}
Since every simple module is reflexive and thus a 2-syzygy module by \ref{2syzy}, it is especially a 1-syzygy module. Thus the subcategory $\Omega^1(mod-A)$ contains every simple module and by \ref{AusReiresult} this subcategory is also extension-closed. Since every finitely generated module can be obtained using extensions from simple modules this gives that $\Omega^1(mod-A)=mod-A$. Thus we have especially: $D(A) \in mod-A = \Omega^1(mod-A)$.
This gives that there exists a short exact sequence:
$$0 \rightarrow D(A) \rightarrow P \rightarrow M \rightarrow 0,$$
where the module $P$ is projective.
But since $D(A)$ is injective, this short exact sequence splits and thus $D(A)$ is a direct summand of the projective module $P$. This shows that every injective module is projective and thus $A$ is selfinjective.
\end{proof}
The previous proposition applies especially to all QF-3 algebras:
\begin{corollary}
A QF-3 algebra is selfinjective in case every simple module is reflexive.
\end{corollary}

\begin{proposition}
Let $A$ be a QF-2 algebra such that each simple $A$-module is reflexive. Then $A$ is selfinjective.
\end{proposition}
\begin{proof}
Since every simple module $S$ is reflexive, we have $Hom_A(S,A) \neq 0$. This shows that the socle of $A$ contains for each indecomposable projective module $P$ exactly one copy of the corresponding simple module $top(P)$ since $A$ is QF-2. This means that the injective hull of $A$ is equal to $D(A)$ and thus $A$ embedds into $D(A)$. But $A$ and $D(A)$ have the same length and thus this embedding has to be an isomorphism, showing $A \cong D(A)$ and thus $A$ is selfinjective.
\end{proof}

Now we look at the problem for local algebras. Note that a local algebra is selfinjective if and only if its socle is simple. A local algebra has a unique simple modules $S$.
We prove it first for commutative algebras, where one even can calculate $S^{**}$ explicitly:
\begin{proposition} \label{commsimple}
Let $A$ be a commutative algebra that is not selfinjective. Let $soc(A)=nS$ for some $n \geq 2$. Then $S^{**} \cong n^2 S$.
\end{proposition}
\begin{proof}
Note that we can assume that our algebras are local and thus $S \cong A/J$.
We first calculate
$S^{*}=Hom_A(A/J,A)$. Now note that the maps in $Hom_A(A/J,A)$ correspond to the maps $f$ in $Hom_A(A,A)$ such that $f(J)=0$. Since the maps in $Hom_A(A,A)$ are left multiplications by an element $a \in A$, $Hom_A(A/J,A)$ is isomorphic to the left annihilator of the module $J$. Now since $A$ is commutative, the left annihilator of $J$ is isomorphic to the socle of $A$, see for example \cite{L2} lemma 8.3. of chapter 1. Thus $S^{*} \cong nS$ and applying $(-)^{*}$ to this isomorphism again, we obtain $S^{**} \cong n S^{*} \cong n^2 S$.
\end{proof}
\begin{corollary}
A commutative algebra $A$ is selfinjective iff every simple module is reflexive.
\end{corollary}
\begin{proof}
We can assume that $A$ is local. $A$ is selfinjective iff $n=1$ in the equation $soc(A)=nS$, which is equivalent to $S^{**} \cong n^2 S$ by the previous proposition.
\end{proof}
\ref{commsimple} motivates the question whether one can in general describe $S^{**}$ in a local algebra.

We give an example that shows that $S^{**}$ does not contain $S$ as a direct summand in general for local algebras. This example was found with the GAP-package QPA, see \cite{QPA}.
\begin{example}
Let $K<x,y>$ denote the non-commutative polynomial ring in $x$ and $y$ and \newline $A:=K<x,y>/(x^3,y^2,xyx)$. Then $A$ has dimension 10 and $S^{**} \cong xA \oplus xA$, where $xA$ is an indecomposable 4-dimensional module.
\end{example}
Now we look at general local algebras. We denote by $soc_l(A)$ the socle of the left regular module.
\begin{proposition}
\begin{enumerate}
\item Let $A$ be an Artin algebra with Jacobson radical $J$ such that every for every simple left module $U$ we have that $U^2$ is a direct summand of the left socle $soc_l(A)$ of $A$ and dually for every simple right module $S$ we have that $S^2$ is a direct summand of the right socle $soc(A)$ of $A$. In case all simple modules are reflexive, $A$ is selfinjective.
\item When the unique simple module in a local algebra is reflexive, the algebra is selfinjective.
\end{enumerate}
\end{proposition}
\begin{proof}
\begin{enumerate}
\item Let $S$ be a simple reflexive module and $M:=Hom_A(S,A)$. Note that the left module $M$ is non-zero since $S$ is assumed to be reflexive. Let $U$ be simple summand of $top(M)$. Then there are at least two maps $M \rightarrow soc_l(A)$ given by mapping the top of $M$ into the direct summand of $U^2$ of $soc_l(A)$ since $U^2$ is a summand of $soc_l(A)$ by assumption. This shows that the length of $S^{**}=M^{*}=Hom_A(M,A)$ is at least two and thus can not be simple.
\item This follows immediately from (1) since a local algebra is selfinjective if and only if the socle of the regular module $A$ is simple.

\end{enumerate}
\end{proof}

As a final result we give a positive answer to the question for algebras $A$ with $A \cong A^{op}$.

\begin{proposition}
\begin{enumerate}
\item Let $A$ be an algebra such that every simple left and every simple right $A$-module is reflexive. Then $A$ is selfinjective.
\item Let $A$ be an algebra with $A \cong A^{op}$ and assume that every simple $A$-module is reflexive. Then $A$ is selfinjective.
\end{enumerate}
\end{proposition}
\begin{proof}
\begin{enumerate}
\item Let $S$ be a simple A-module and $M:=Hom_A(S,A)$. 
We have two cases: \newline
\underline{Case 1}: $M$ is not semisimple. Then $M$ is a torsionless left $A$-module and thus there is an embedding $M \rightarrow A^n$ for some $n$. Then there is a map $M \rightarrow A^n \rightarrow A$ where the first map is the inclusion of $M$ into $A^n$ and the second map a projection, whose image is not included in the left socle of $A$.
On the other hand there is a map $M \rightarrow A$ whose image is included in the left socle of $A$ (by mapping the top of $M$ into some simple socle summand).
Thus in this case $S$ is not isomorphic to $S^{**}$ because $S^{**}=Hom_A(M,A)$ has length at least two. This gives a contradiction. \newline
\underline{Case 2}: $M$ is semisimple. Then $M$ has to be simple in fact or else $S^{**}=M^{*}$ is not indecomposable (here we use that all simple left modules are reflexive and thus $U^{*}$ is non-zero for any simple left $A$-module $U$).
This implies that $Hom_A(S,A)$ is simple for all $S$ and thus every simple module is exactly once a direct summand of $soc(A)$. This implies that $A/J=soc(A)$, which is equivalent to $A$ being selfinjective.

\item This follows immediately from (1).

\end{enumerate}

\end{proof}


\begin{thebibliography}{Gus}
\bibitem[AnFul]{AnFul} Anderson, F.; Fuller, K.: {\it Rings and Categories of Modules.} Graduate Texts in Mathematics, Volume 13, Springer-Verlag, 1992. 
\bibitem[AR]{AR} Auslander, M.; Reiten, I.: {\it Syzygy Modules for Noetherian Rings.} Journal of Algebra, Volume 183, Issue 1, 1 July 1996, Pages 167-185.
\bibitem[ARS]{ARS} Auslander, M.; Reiten, I.; Smalo, S.: {\it Representation Theory of Artin Algebras.}
Cambridge Studies in Advanced Mathematics, Volume 36, Cambridge University Press, 1997.
\bibitem[Che]{Che} Chen, X.: {\it Gorenstein Homological Algebra of Artin
Algebras.} https://arxiv.org/abs/1712.04587.
\bibitem[GH]{GH} Gubareni, N.; Hazewinkel, M.: {\it Algebras, Rings and Modules, Volume 2: Non-commutative Algebras and Rings.} CRC Press, 2017.
\bibitem[L]{L} Lam, T. Y. : {\it Lectures on modules and rings.} Graduate Texts in Mathematics, Springer, 1998.
\bibitem[L2]{L2} Landrock, P.: {\it Finite group algebras and their modules.} London Mathematical Socity Lecture Note Series 84, 1983.
\bibitem[QPA]{QPA} The QPA-team, QPA - Quivers, path algebras and representations - a GAP package, Version 1.25; 2016 (https://folk.ntnu.no/oyvinso/QPA/)

\end{thebibliography}
\end{document}